\documentclass[12pt, a4paper]{article}

\usepackage{amssymb,amsfonts,latexsym}
\usepackage{amsmath,amsthm}
\usepackage{mathrsfs}

\usepackage{hyperref}

\usepackage[all]{xy}

\usepackage[active]{srcltx}

\usepackage{epic,eepic,eepicemu}
\usepackage[usenames]{color}

\usepackage[normalem]{ulem}

\usepackage{enumerate}

\allowdisplaybreaks

\numberwithin{equation}{section}

\definecolor{purple}{rgb}{0.49, 0.06, 0.51}

\def\R{\mathbb{R}}

\def\C{\mathbb{C}}

\def\s{\sigma}

\DeclareMathOperator{\Aut}{Aut}
\DeclareMathOperator{\Gal}{Gal}

\newcommand{\ox}{\otimes}

\newcommand{\id}{\mathrm{id}}

\DeclareMathOperator{\Sym}{Sym}

\DeclareMathOperator{\Trd}{\mathrm{Trd}}

\DeclareMathOperator{\Char}{\mathrm{char}}

\newcommand{\bbar}{\overline{\phantom{x}}}

\newtheorem{thm}{Theorem}[section]
\newtheorem{prop}[thm]{Proposition}
\newtheorem{cor}[thm]{Corollary}
\newtheorem{lemma}[thm]{Lemma}
\theoremstyle{definition}

\begin{document}
\title{Galois extensions, positive involutions and an application to unitary space-time coding}
\author{Vincent Astier and Thomas Unger}

\date{}
\maketitle

\begin{abstract}
We show that under certain conditions every maximal symmetric subfield of a central division algebra with positive unitary involution 
$(B,\tau)$
will be a Galois extension of the fixed field of $\tau$ and will ``real split'' $(B,\tau)$. As an application we show that a sufficient condition
for the existence of positive involutions on certain crossed product division algebras, considered by Berhuy in the context of unitary
space-time coding, is also necessary, proving that Berhuy's construction is optimal.
\smallskip

\smallskip

\end{abstract}

\section{Introduction}

Unitary space-time codes  have been studied in the context of non-coherent space-time coding,
used in MIMO signalling where the receiver has no knowledge of the properties of the transmission channel. 

In \cite{Berhuy-2014} Berhuy provides a systematic construction of unitary codes  based on central simple 
algebras with unitary involution. In particular, division
algebras with unitary involution result in codes that are fully diverse, a desirable property. 

The basic idea is as follows: let $k/k_0$ be a quadratic extension of number fields whose nontrivial automorphism is given by
complex conjugation $\bbar$ and let $(B,\tau)$ be a central simple $k$-algebra with unitary
$k/k_0$-involution (i.e., $\tau|_k=\bbar$). 
Let $K$ be a subfield of $\C$ that contains $k$ and with the property that it splits $B$, i.e., such that $B\ox_k K \cong 
M_n(K)$. (Such a
field always exists, cf. \cite[Cor.~IV.1.9]{B-O-2013}.) Under this isomorphism, the unitary involution $\tau\ox \bbar$ is transported to a unitary 
involution $\tau'$ on $M_n(K)$.
If $\tau'$ were to be conjugate transposition $*$ (Berhuy calls the involution $\tau$ positive definite in this case), then  
unitary matrices in $M_n(K)$ can be considered, corresponding to unitary elements in $B$
and Berhuy shows how to construct suitable codes from these. 

The problem of constructing fully diverse unitary space-time codes is thus reduced to the problem of constructing 
division algebras with positive definite unitary involution. For practical purposes such constructions need to be explicit.

We recall Berhuy's construction:

\begin{prop}[{\cite[Def.~2.3, Lemma~2.6, Ex.~3.9]{Berhuy-2014}}]\label{prop_ber}
Let $k/k_0$ be a quadratic extension of number fields whose nontrivial automorphism is given by
complex conjugation.
Consider the crossed product algebra $B=(\xi, L/k,G)$, where $L/k$ is a finite Galois extension with Galois group $G$
and $\xi$ is a $2$-cocycle of $G$ with values in $L$. 
Assume that $L\subseteq \C$ and   that complex conjugation induces an automorphism $\alpha$ of $L$.
Assume that $\alpha(\xi_{\s,\rho})\xi_{\s,\rho}=1$
for all $\s,\rho \in G$.
If 
\begin{equation}\label{ber}
\alpha \circ \s =\s \circ \alpha \quad\text{ for all } \s \in G,
\end{equation}
 then there exists a unitary  involution $\tau$ on $B$ that acts as inversion
on the generators of $B$ and such that $\tau|_L=\alpha$. Furthermore, $\tau$ is unique with these properties and is positive definite.
\end{prop}

Berhuy also provides  examples of such algebras $B$ that are division algebras, cf. \cite[Examples~3.11, 4.8, 4.14, 4.15]{Berhuy-2014}.

In \cite{A-U-pos} we made a detailed study of positive (definite) involutions on algebras with involution. In 
Section~\ref{sec2} we consider division algebras with positive unitary involution and show that under certain conditions
every maximal symmetric subfield will be a Galois extension of the base field and will  ``real split'' the algebra with involution.

As an application we show in Section~\ref{sec3} that 
in the division algebra case, if the crossed product $B$ carries a positive definite unitary involution $\tau$, then 
$\tau|_L$ commutes with all elements of $G$ and $\tau|_L=\alpha$. As a consequence Berhuy's construction is optimal
in the sense that \eqref{ber} is both necessary and sufficient for the existence of a positive definite unitary involution 
on $B$.

\section{Galois extensions and positive involutions}\label{sec2}

In this section we use the following notation: $k$ is a field of characteristic not $2$, $B$ is a central division 
$k$-algebra of degree $n$ with unitary
involution $\tau$ and $k_0$ is the fixed field of $\tau$. Then $k/k_0$ is a quadratic extension and $\tau|_k$ is the 
unique nontrivial $k_0$-automorphism of $k$.  
Let $d\in k_0$ be such that $k=k_0(\sqrt{-d})$. Note that $\tau(\sqrt{-d})=-\sqrt{-d}$.
The 
space of $\tau$-symmetric elements of $B$ is denoted by $\Sym(B,\tau)$. 
We denote the space of orderings of $k_0$ by $X_{k_0}$. 
Note that $X_{k_0} \not=\varnothing$ implies that the characteristic of $k_0$ is zero.
If $P\in X_{k_0}$ we denote by $(k_0)_P$ a real closure of $k_0$ at $P$.
The involution $\tau$ is positive (definite) at an
ordering $P\in X_{k_0}$ if the associated involution trace form $\Trd_B(\tau(x)x)$ is positive definite at $P$,
cf. \cite[\S 4]{A-U-PS} (or \cite[Thm.~4.19]{Berhuy-2014} for algebras with unitary involution over number fields).

In \cite[Thm.~8.5, Def.~8.6, Prop.~8.8]{A-U-pos} we showed that if $\tau$ is positive at $P$,
then 

\begin{enumerate}[(a)]
\item  any maximal symmetric
subfield $K$ of $(B,\tau)$ that contains $k_0$ real splits $(B,\tau)$, i.e., $P$ extends to $K$ and
  \[(B\ox_{k_0} K, \tau\ox \id) \cong (M_n(K(\sqrt{-d})),  *),\]
  with $d\in P$ and $*$ conjugate transposition; 

\item there  exists a finite Galois extension of $k_0$ that real splits $(B,\tau)$.
\end{enumerate}

The main result in this section is that any maximal symmetric subfield  of $(B,\tau)$ that contains $k_0$ and that 
satisfies  additional assumptions will be a Galois extension of $k_0$ (that then real splits $(B,\tau)$).

\begin{lemma}\label{max|F}
Let $L_0$ be a finite field extension of $k_0$ and let $L=L_0(\sqrt{-d})$.
If $L/k$ is Galois, then $L/k_0$ is also Galois.
\end{lemma}

\begin{proof}
  For every $f \in
  \Gal(L/k)$, define $f',f'' : L \rightarrow L$
  by $f'(a+b\sqrt{-d}) = f(a)+f(b)\sqrt{-d}$ and $f''(a+b\sqrt{-d}) =
  f(a)-f(b)\sqrt{-d}$, where $a,b\in k_0$. A direct computation shows that $f,f' \in
  \Aut(L/k_0)$, so that $|\Aut(L/k_0) |= 2[L:k] =
  [L:k_0]$, proving that $L/k_0$ is Galois.
\end{proof}

\begin{prop}\label{posinvo}
 Let $P\in X_{k_0}$. The following
  statements are equivalent:
    \begin{enumerate}[$(1)$]
      \item $B$ carries a unitary involution that is positive at $P$;
      \item $d \in P$, i.e., $P$ does not extend to $k=k_0(\sqrt{-d})$. 
    \end{enumerate}
\end{prop}

\begin{proof}
  By \cite[Thm.~6.8 and Prop.~8.4]{A-U-pos}.
\end{proof}

Note that the involution in Proposition~\ref{posinvo}(1) is not necessarily $\tau$.

\begin{prop}\label{Galois-2} 
Let $L_0$ be a maximal subfield of $\Sym(B,\tau)$ containing $k_0$. Assume that
$L_0(\sqrt{-d})/k$ is Galois.
If $\tau$ is positive at some ordering $P$ of $k_0$, then $L_0/k_0$ is Galois.
\end{prop}

\begin{proof}
Let $a \in L_0$ be such that $L_0 = k_0(a)$. Since $\tau(a) = a$, 
we have  in $B \ox_{k_0} (k_0)_P$ that $(\tau \ox \id)(a \ox 1) = a \ox 1$. 
Since $\tau$ is  positive at $P$, we have that $d\in P$ by Proposition~\ref{posinvo}. Let $Q$ be an extension
of $P$ to $L_0$, cf. \cite[Thm.~8.5]{A-U-pos}
and note that we may thus take $(L_0)_Q=(k_0)_P$.
Consider the isomorphisms of algebras with involution
\begin{align*}
(B\ox_{k_0} (k_0)_P, \tau\ox \id) &\cong
(B\ox_{k_0} (L_0)_Q, \tau\ox \id)\\ 
& \cong (B\ox_{k_0} L_0, \tau \ox \id) \ox_{L_0} ((L_0)_Q, \id)\\
&\cong (M_n(L_0(\sqrt{-d})), *) \ox_{L_0} ((L_0)_Q, \id)\\ 
&\cong (M_n((L_0)_Q(\sqrt{-1})), *)\\
&\cong (M_n((k_0)_P(\sqrt{-1})), *),
\end{align*}
where the third isomorphism follows from (a).
Let $M_a \in M_n((k_0)_P(\sqrt{-1})) $ be the matrix
corresponding to $a \ox 1$  under these isomorphisms
and note that  ${M_a}^* = M_a$ since $\tau(a)=a$.
As in the proof
  of \cite[Prop.~8.8]{A-U-pos}, the roots of the minimal polynomial $\min_a$ 
  of $a$ over $k_0$
  are all in $(k_0)_P$. They are all also
in $L_0(\sqrt{-d})$ since $L_0(\sqrt{-d}) / k_0$ is Galois by Lemma~\ref{max|F} and $a\in L_0$. 
  Therefore all roots of $\min_a$ are in
   $(k_0)_P\cap L_0(\sqrt{-d}) =L_0$, $\min_a$ splits in $L_0$, and thus $L_0/k_0$
  is Galois (since $\Char k_0=0$).
\end{proof}

\begin{cor} \label{thecor}
With the same hypotheses as Proposition~\ref{Galois-2}, assume that $\tau$ is positive at
some ordering $P$ of
$k_0$. Let $L=L_0(\sqrt{-d})$.
Then $\tau|_{L} \circ \gamma = \gamma \circ \tau |_{L}$ for every $\gamma \in
      \Gal(L/k_0)$, and so   $\tau|_{L} \circ \gamma = \gamma \circ \tau |_{L}$ for every $\gamma \in
      \Gal(L/k)$.
\end{cor}

\begin{proof} Observe that $\tau|_L  \in \Gal(L/k_0)$.
 Since $L_0$ is the fixed field of $\tau|_L$ and $L_0/k_0$
is Galois by Proposition~\ref{Galois-2}, the subgroup generated by $\tau|_L$ in
$\Gal(L/k_0)$,
  $\{\id, \tau|_L\}$ is normal in $\Gal(L/k_0)$. The first statement follows immediately. 
  The second statement follows from Lemma~\ref{max|F} since $\Gal(L/k) \subseteq \Gal(L/k_0)$.
\end{proof}

\section{Application to Berhuy's construction}\label{sec3}

Assume now that
$k/k_0$ is a quadratic extension of number fields whose nontrivial automorphism is given by complex conjugation $\bbar$. 
In this case  $k_0$
will be a subfield of $\R$ and we  will only consider the unique ordering on $k_0$ induced by the standard 
ordering on $\R$. Furthermore, $k=k_0(\sqrt{-d})$ with $d>0$.

\begin{prop} \label{prop3.1}
 Let $B=(\xi,L/k,G)$
be a crossed product division algebra, where $L/k$ is a finite Galois extension with Galois group $G$ and $\xi$ is a $2$-cocycle of $G$ with
values in $L$. 
Assume that $L\subseteq \C$ and   that complex conjugation induces an automorphism $\alpha$ of $L$.
Assume that $\alpha(\xi_{\s,\rho})\xi_{\s,\rho}=1$
for all $\s,\rho \in G$.
If $B$ carries a positive unitary involution $\tau$ such that $\tau|_L=\alpha$,
then 
$\tau|_L$  commutes with all elements of $G$.
\end{prop}

\begin{proof}
Since $B$ is a crossed product algebra, $L$ is a maximal subfield of $B$, cf. \cite[Prop.~VI.2.1]{B-O-2013}. Let $L_0$ be the fixed field of $\alpha=\tau|_L$. Then $[L:L_0]=2$ and
$L_0$ is a maximal subfield of $\Sym(B,\tau)$. Since $k_0(\sqrt{-d})=k\subseteq L$, we have $k_0\subseteq L_0$ and $L_0(\sqrt{-d})=L$. 
The conditions of Proposition~\ref{Galois-2} are thus satisfied and so $\alpha$ commutes with all the elements of $G$ by Corollary~\ref{thecor}.
\end{proof}

As a consequence of Propositions~\ref{prop_ber} and \ref{prop3.1} we obtain the optimality of Berhuy's construction:

\begin{cor}
Let $B=(\xi,L/k,G)$ be as in Proposition~\ref{prop3.1}. Then $B$ carries a positive unitary involution $\tau$ if and only if
$\tau|_L$ commutes with the elements of $G$.
\end{cor}

\def\cprime{$'$}

\textsc{School of Mathematics and Statistics, University College Dublin, Belfield,
Dublin~4, Ireland} 

\emph{E-mail address: } \texttt{vincent.astier@ucd.ie, thomas.unger@ucd.ie}

\end{document}